\documentclass[11pt]{amsart}
\usepackage{graphicx}
\usepackage{epstopdf}
\usepackage{url}
\usepackage{latexsym}
\usepackage{amsfonts,amsmath,amssymb}
\newtheorem{theorem}{Theorem}[section]
\newtheorem{proposition}[theorem]{Proposition}
\newtheorem{lemma}[theorem]{Lemma}

\newtheorem{corollary}[theorem]{Corollary}

\newtheorem{example}[theorem]{Example}

\theoremstyle{remark}
\newtheorem{remark}{Remark}

\date{\today}

\begin{document}

\title{Sums of $k$-potent matrices}

\author{Ivan Gargate}
\address{UTFPR, Campus Pato Branco, Rua Via do Conhecimento km 01, 85503-390 Pato Branco, PR, Brazil}
\email{ivangargate@utfpr.edu.br}

\author{Michael Gargate}
\address{UTFPR, Campus Pato Branco, Rua Via do Conhecimento km 01, 85503-390 Pato Branco, PR, Brazil}
\email{michaelgargate@utfpr.edu.br}

\begin{abstract}
We study sums of $k$-potent matrices. We show the conditions by which a complex matrix $A$ can be expressed as a sums of $k$-potent matrices. Also we obtain conditions by which a complex matrix $A$ can be expressed as a sum of finite order elements. This generalize some results obtain by Wu in \cite{Wu}. Also we study the sum of $k$-potent matrices in $\mathcal{M}_{\mathcal{C}f}(F)$ with $F$ be a field and proof that  any matrix in this space can be expressed as a sum of $14$ $(k+1)$-potent matrices preserving the result obtain by Slowik in \cite{SlowikIdemp}.
\end{abstract}



\maketitle

\section{Introduction}\label{intro}
It is a classical question whether the elements of a ring or a group can be expressed as sums or products of elements of some particular set.

In the case of expressed matrices as a product, there are several authors who have works in this respect, for example see \cite{Halmos,Slowik-1,Zheng,Hou,Halmos2}.

Recently the authors in \cite{IvanMichael} gave a way to express matrices with product of commutators of finite order.

In the case of expressed matrices as product of idempotents we can see \cite{Wu}. In this context, there are various generalizations of concept of matrices idempotents, for instance see \cite{Mevlin}. In this paper, the authors explored the  $k$-potent matrices using the definition in \cite{Huang2} or the case $(1,k)$ in \cite{Mevlin} and generalize the work done by Wu and Slowik in \cite{Wu} and  \cite{SlowikIdemp} respectively  and show how to express matrices as sum of $k$-potent matrices.

Wu in \cite{Wu} shows a necessary and sufficient condition for a complex square matrix be the sum of finitely many idempotent matrices. In the section 2 the authors generalize this work proving the following result:

\begin{theorem}\label{th1}  A complex square matrix $A$ is a sum of finitely many $(k+1)$-potent if and only if there is an polynomial $F(x)$ with coefficient integers non-negatives and degree $\leq k-1$ such that $trA)=F(\omega)$ and $F(1)\geq rank(A)$.
\end{theorem}

And in the geral case, we show the following:

\begin{theorem}\label{th4} Let $A$ be a complex square matrix and suppose that
\begin{itemize}
    \item [1.)] There are $\omega_1,\omega_2,\cdots,\omega_n$ and $\beta_1,\beta_2,\cdots,\beta_n$ integer numbers with  $\omega_i\neq \omega_j$ and $\beta_i\neq \beta_j$ for $i\neq j$ such that $\omega_i$ is a $\beta_ith $ root of unity for $i=1,2,\cdots,n$.
    \item[2.)] There are polynomials $F_1(x),F_2(x),\cdots, F_n(x)$ with positive coefficients integers and $1\leq degree(F_i) \leq \beta_i-1$, such that $$tr(A)=a_0+F_1(\omega_i)+F_2(\omega_2)+\cdots + F_n(\omega_n)$$
    where $a_0$ is a positive integer number.
    \end{itemize}
    Then, $A$ is a sum of finitely many $\beta_1+1,\beta_2+1,\cdots,\beta_n+1$-potent matrices if and only if $$a_0+F_1(1)+F_2(1)+\cdots+F_n(1)\geq rank(A).$$
 \end{theorem}

And we enunciated some remarks about linear combinations of these special matrices.

In the section 3, the authors generalize the results of Slowik in \cite{SlowikIdemp} for  the case of $k$-potent matrices on $\mathcal{M}_{\mathcal{C}f}(F)$, proving the following result:

\begin{theorem}\label{th2}
Let $F$ any field. Any matrix $A \in\mathcal{M}_{\mathcal{C}f}(F),$  can be expressed as a sum of at most 14,   $(k+1)$-potent  matrices from $\mathcal{M}_{\mathcal{C}f}(F).$
\end{theorem}

\section{Sums of k-potent  complex square Matrices}
In this section we consider matrices with complex entries or complex matrices. We start introducing the notation. Let $k\geq 2$ be an integer. If a square matrix $A$ satisfies $A^k = A$, then $A$ is said to be $k$-potent (see definition in \cite{Huang}). For $k=2$, $A$ is said idempotent. 
A matrix $A$ is called of order $k$ if $A^k = I$, in the case that $k = 2$, $A$ is called an involution.

In the following, $tr(A)$ denotes the trace of a matrix $A$, $ran(A)$ denotes its range, $rank(A)$ the dimension of $ran(A)$, and $ker(A)$ the kernel of $A$. Two matrices $A$ and $B$ are similar if $XA = BX$ for some nonsingular matrix $X$, they are unitarily equivalent if the above $X$ can be chosen to be unitary. If $A$ and $B$ act on spaces $H$ and $K$, respectively, then
$$A\oplus B=\left[ \begin{array}{cc}
A & 0 \\
0 & B
\end{array}
\right]
$$
acts on $H\oplus K$, the orthogonal direct sum of $H$ and $K$.

Let $\omega \in F$ be a kth root of unity with $\omega \neq 1$, then we have the following lemma:



\begin{lemma}\label{lemapoli}
Let $A$ be an matrix $(k+1)$-potent and  $\omega$ the $k$-root of unity. There is a unique polynomial $F(x)$ with non-negative  integer coefficients  and degree $\leq k-1$ such that $tr(A)=F(\omega)$ and $rank(A)=F(1).$ 
\end{lemma}
\begin{proof} Consider a $(k+1)$-potent matrix $A$, then their possible eigenvalues satisfies the equation $x^{k+1}-x=0$, so they can assume the values $0,1,\omega,\omega^2, \cdots, \omega^{k-1}$ where $\omega$ is the $k$-root of unity. Then $A$ is triangularizable and by similarity,  the $tr(A)$ is a polynomial in $\omega$ with degree $\leq k-1$ and define this polynomial as $F(x)$. The coefficient of the variable $x^j$ in this polynomial represented the number of times that appears $w^j$ in the diagonal of the triangularization of $A$. So $F(1)$ represented the number of non zero elements that appears in the diagonal of the triangularization of $A$ and this number is the $rank(A)$.  
\end{proof}

In the following Lemma we show basic $(k+1)$-potent matrices.

\begin{lemma}\label{lema4} Let $\omega\in F$ be a kth root of unity. The matrices of the form

$$A= \left[\begin{array}{ccccccc} 0 & & & & & &  \\  & \ddots & & \vdots& & &  \\
& &0  & x & & & \\
& & & \omega & & & \\
& & & y & 0 & & \\
& & & \vdots & & \ddots & \\
& & & & & & 0\end{array}\right]$$
satisfies $X^k\neq X$ and $X^{k+1}=X$. So  $A$ is a $(k+1)$-potent matrix.
\end{lemma}
\begin{proof}
For any $j$ positive integer  we have that 
$$A^j= \left[\begin{array}{ccccccc} 0 & & & & & &  \\  & \ddots & & \vdots& & &  \\
& &0  & x\omega^{j-1} & & & \\
& & & \omega^j & & & \\
& & & y\omega^{j-1} & 0 & & \\
& & & \vdots & & \ddots & \\
& & & & & & 0\end{array}\right].$$
\end{proof}

In some cases it's not difficult to show that some matrices are sums of $(k+1)$-potent matrices, see for example the folowing Remark. 
\begin{remark}
If $k$ is even, then we can expressed the identity matrix $I_n$ as sum of $(k+1)$-potent matrices, its follows from the property $1+\omega+\omega^2+\cdots+\omega^{k-1}=0$, in this case
$$\begin{array}{rl}I_n=& \displaystyle\sum^{k-1}_{i=1}( -w^i)\cdot I_n \\ = & \displaystyle\sum^{k-1}_{i=1} \sum^{n}_{j=1} \left[\begin{array}{ccccccc} 0 & & & & & &  \\  & \ddots & & \vdots& & &  \\
& &0  & 0 & & & \\
& & & -\omega^i & & & \\
& & &  0 & 0 & & \\
& & & \vdots & & \ddots & \\
& & &j-column & & & 0\end{array}\right]\end{array}$$
Then, $I_n$ can be expressed as sum of $n(k-1)$,  $(k+1)$-potent matrices. Also, consider $A=\alpha I_n$ with $\alpha$ an positive integer number, then $A$ can be expressed as sum of $\alpha n (k-1)$,  $(k+1)$-potent matrices.

\end{remark}
Now we prove our first main result.

\begin{proof}[Proof of Theorem \ref{th1}]
For the case $k=1$ it is proved by Wu in \cite{Wu}, so we consider the case that $\geq 2$. \\
$[\Rightarrow]$ If $A=\sum_{j=1}^m E_j$, where each $E_j$ is $(k+1)$-potent, then by the Lemma \ref{lemapoli} denote by $F_j(x)$ their respective polynomial. Then
$$tr(A)=tr(\sum_{j=1}^mE_j)=\sum_{j=1}^mtr(E_j)=\sum_{j=1}^mF_j(\omega).$$
Denote by $F(x)=\sum_{j=1}^mF_j(x)$ and observe that 
$$F(1)=\sum_{j=1}^mF_j(1)=\sum_{j=1}^m rank(E_j)\geq rank(\sum_{j=1}^mE_j)=rank(A).$$
$[\Leftarrow]$ Let $r=rank(A)$, then $A$ is unitary equivalent to a matrix of the form

$$\left[\begin{array}{cc} A_1 & 0 \\ A_2 & 0\end{array}\right]$$
with respect to the decomposition $H=rank(A^*)\oplus ker(A)$, where $A_1$ is of size $r\times r$. We have two cases to consider

\begin{itemize}
    \item[(1.)] $A_1$ is not a scalar matrix. By hipotesse, there is a polynomial $F(x)=a_0+a_1x+a_2x^2+\cdots+a_{k-1}x^{k-1}$ with integer coefficients and  $a_0,a_1,\cdots,a_{k-1}> 0$ . By the condition $F(1)=a_1+a_2+\cdots+a_{k-1}\geq r=rank(A)$ we can choose integer numbers $r_0,r_1,\cdots, r_{k-1}$ such that $a_i\geq r_i>0$ and $r_0+r_1+\cdots+r_{k-1}=r.$ Then for each $j=0,1,\cdots, k-1$, since $a_j\geq r_j$ then we choose positive integers $a_j^{(1)},a_j^{(2)},\cdots, a_j^{(r_j)}$ such that $a_j=a_j^{(1)}+a_j^{(2)}+\cdots + a_j^{(r_j)}$.
    Appliying this decomposition, we have that
    $$tr(A)= \sum^{r_0}_{i=1} a_0^{(i)}+ \left(\sum^{r_1}_{i=1} a_1^{(i)}\right)\omega+ \cdots +  
    \left(\sum^{r_{k-1}}_{i=1} a_{k-1}^{(i)}\right)\omega^{k-1}.$$
    By the Theorem 2 of Fillmore in \cite{Fillmore}, $A_1$ is similar to a matrix $B$  of the form
    
    $$\left[\begin{array}{cccc} 
    \left[\begin{array}{ccc} 
    a_0^{(1)} & &  \\  & \ddots & \\  & & a_0^{(r_0)}\end{array}\right]
    & & & * \\ 
    & \left[\begin{array}{ccc} a_1^{(1)}\omega & &  \\  & \ddots & \\  & & a_1^{(r_1)}\omega \end{array}\right] & & \\
    & & \ddots &  \\
    * & & & \left[\begin{array}{ccc} a_{k-1}^{(1)}\omega^{k-1} & &  \\  & \ddots & \\  & &  a_{k-1}^{(r_{k-1})}\omega^{k-1} \end{array}\right]
    \end{array}\right]$$

    Then, $A$ is similar to 
    
    $$C=\left[\begin{array}{c|c} B & 0 \\ \hline * & 0 \end{array}\right]$$
    
    Denote by $X_0^{(1)},X_0^{(2)},\cdots, X_0^{(r_0)}, X_1^{(1)},X_1^{(2)},\cdots,X_1^{(r_1)}, \cdots, X_{k-1}^{(1)},X_{k-1}^{(2)},$ $\cdots,X_{k-1}^{(r_{k-1})}$ the first $r=r_0+r_1+\cdots+r_{k-1}$ column vectors of $C$, then $C$ can be expressed as
    
    $$\begin{array}{rl} C=& \displaystyle\sum^{r_0}_{i=1}\left[ \begin{array}{ccccccc} 0& \cdots &0  & X_0^{(i)}  & 0 & \cdots & 0\end{array} \right]  \\ &+\displaystyle \sum^{r_1}_{i=1}\left[ \begin{array}{ccccccc} 0& \cdots &0  & X_1^{(i)}  & 0 & \cdots & 0\end{array} \right] \\
    &+ \cdots+ \displaystyle \sum^{r_{k-1}}_{i=1}\left[ \begin{array}{ccccccc} 0& \cdots &0  & X_{k-1}^{(i)}  & 0 & \cdots & 0\end{array} \right]\end{array}$$
    where, the column $X_j^{(i)}$ appears in the position  $r_0+r_1+\cdots+r_{j-1}+i$, or also
    
      $$\begin{array}{rl} C=& \displaystyle\sum^{r_0}_{i=1}a^{(i)}_0\left[ \begin{array}{ccccccc} 0& \cdots &0  & \displaystyle\frac{1}{a_0^{(i)}}X_0^{(i)}  & 0 & \cdots & 0\end{array} \right]  \\ &+\displaystyle \sum^{r_1}_{i=1}a_1^{(i)}\left[ \begin{array}{ccccccc} 0& \cdots &0  & \displaystyle\frac{1}{a_1^{(i)}}X_1^{(i)}  & 0 & \cdots & 0\end{array} \right] \\
    &+ \cdots+ \displaystyle \sum^{r_{k-1}}_{i=1}a_{k-1}^{(i)}\left[ \begin{array}{ccccccc} 0& \cdots &0  & \displaystyle\frac{1}{a_{k-1}^{(i)}}X_{k-1}^{(i)}  & 0 & \cdots & 0\end{array} \right]\end{array}
    $$
    
    Observe that, for $j>0$, the matrices 
    $$\left[ \begin{array}{ccccccc} 0& \cdots &0  & \displaystyle\frac{1}{a_j^{(i)}}X_j^{(i)}  & 0 & \cdots & 0\end{array} \right]$$
    are in the form of Lemma \ref{lema4} and from here we have that these matrices are $(k+1)$-potent and for $j=0$, follows from Wu in \cite{Wu} that these respective matrices are idempotents, so $(k+1)$-potent matrices. Then we have expressed $C$ as a sum of $\sum^{r_0}_{i=1}a_0^{(i)}+\sum^{r_1}_{i=1}a_1^{(i)}+\cdots+ \sum^{r_{k-1}}_{i=1}a_{k-1}^{(i)}=F(1)$ matrices $(k+1)$-potent and by similarity we concluded that they also hold for $A$.

    \item[(2.)] If $A_1$ is a scalar matrix, then $tr(A)=F(1)=r$ and by Wu in \cite{Wu} we have that $A$ is sum of $r$ idempotent matrices and so $(k+1)$-potent matrices.
\end{itemize}
\end{proof}

Now we prove our second main result.


\begin{proof}[Proof of Theorem \ref{th4}]
    The necessity is similar to the proof of Theorem \ref{th1}.
    For the converse, let $r=rank(A)$ then $A$ is unitary equivalent to matrix of the form 
    
    $$\left[\begin{array}{cc} A_1 & 0 \\ A_2 & 0\end{array}\right]$$
    with $A_1$ an matrix of size $r\times r$. Choose positive integer numbers $r_0,r_1,\cdots, r_n$ such that $a_0\geq r_0$, $F_i(1)\geq r_i$ for $i=1,2,\cdots,n$ and $r_0+r_1+\cdots+r_n=rank(A)$.
    
    Proceeding similar to the proof of Theorem \ref{th1} we have that $A_1$ is similar to an matrix 
    
    $$B=\left[\begin{array}{cccc}B_0 & & & * \\ & B_1 & & \\ & & \ddots & \\ * & & & B_n\end{array}\right]$$
    such that each $B_i$ is a $r_i\times r_i$ matrix with $tr(B_0)=a_0$ and $tr(B_i)=F_i(\omega_i)$
 for $i=1,2,\cdots,n$. Observe that
 $$B=\left[\begin{array}{cccc}B_0 & 0 & & 0 \\ *  & 0 & & 0 \\ &   & \ddots & \\ * & 0 & & 0\end{array}\right]+\left[\begin{array}{cccc}0& * & & 0 \\ 0 & B_1 & & 0 \\ & & \ddots & \\ 0 & * & & 0\end{array}\right]+\cdots+\left[\begin{array}{cccc}0 & 0 & & * \\  0 & 0 & & * \\ & & \ddots & \\ 0 & 0  & & B_n\end{array}\right] $$
 and applying the Theorem \ref{th1} in each to the matrices $B_i$ we obtain the result.
 
    \end{proof}

Wu in \cite{Wu} gave a result that  $A$ is idempotent if and only if $2A-I$ is an involution. Unfortunately, if $k>2$ this equivalence is not know but in the following Lemmas we show various forms of construct these special matrices from the other.
\begin{lemma}
If $A$ is an matrix of order $k$ then the matrix $$B=\displaystyle\frac{1}{k}(I+A+A^2+\cdots +A^{k-1}),$$ is $k$-potent. 
\end{lemma}
\begin{proof}
Its can proof using simple calculations.
\end{proof}
In the following Lemma we show how its possible construct new matrices of order $k$ using $k$-potent  matrices.
\begin{lemma} Let $A$ be an matrix $k$-potent and $\omega$ a kth root of unity, then:
\begin{itemize}
    \item [1.)] $\omega-1$ is an solution of the equation \begin{equation} x^{k-1}+\binom{k}{1}x^{k-2}+\binom{k}{2}x^{k-3}+\cdots + \binom{k}{k-2}x+\binom{k}{k-1}=0 \label{eq1}\end{equation}
    and the matrix $B=(\omega-1)A^{k-1}+I$ is an matrix of order $k$.
    \item[2.)] If $k$ is even, then $1+\omega$ is solution of the equation
    \begin{equation} x^{k-1}-\binom{k}{1}x^{k-2}+\binom{k}{2}x^{k-3}-\cdots + \binom{k}{k-2}x-\binom{k}{k-1}=0 \label{eq2}\end{equation}
    and the matrix $C=(1+\omega)A^{k-1}+I$ is an matrix of order $k$.
\end{itemize}

\end{lemma}
\begin{proof} Since that $\omega^k=1$ and using the Newton binomial formula, then
$$\begin{array}{rl} 1=&\omega^k \\ 1=&(\omega-1+1)^k\\ 1=& \displaystyle \sum^{k}_{i=0}\binom{k}{i} (\omega-1)^i \\
1=&\displaystyle\sum^{k}_{i=1} \binom{k}{i} (\omega-1)^i+1 \\
0=&(\omega-1)\left[\displaystyle\sum^{k}_{i=1}\binom{k}{i} (\omega-1)^{i-1}\right],
\end{array}$$
So, $\omega-1$ satisfies the equation (\ref{eq1}). Therefore, we observe that, for $i=2,3,\cdots,k-1$ $(A^{k-1})^i=A^{(k-1)i}=A^{(ki-i)}=A^{(i-1)k+k-i}=(A^k)^{(i-1)}\cdot A^{k-i}=A^{i-1}\cdot A^{k-i}=A^{k-1},$ and using the Newton binomial formula
$$\begin{array}{ll}B^k& = (\alpha A^{k-1}+I)^k \\ &= \displaystyle \sum^{k}_{i=0}\binom{k}{i} (\alpha A^{k-1})^i 
\\ & =\displaystyle \sum^{k}_{i=1}\binom{k}{i}\alpha^i A^{k-1} + I \\
& = \alpha\left(\displaystyle \sum^{k}_{i=1}\binom{k}{i}\alpha^{i-1}\right)+I 
\\ & = I,\end{array}$$ with $\alpha=\omega-1$.
The proof of item (2) is similar.
\end{proof}

Using the simple matrices of Lemma \ref{lema4} its possible construct  $(k+1)$-potent matrices.

\begin{lemma}\label{lemafiniteorder} Let $\omega$ be a kth root of unity and $k$ an even number. With the matrices of the form 
$$A= \left[\begin{array}{ccccccc} 0 & & & & & &  \\  & \ddots & & \vdots& & &  \\
& &0  & x & & & \\
& & & \omega & & & \\
& & & y & 0 & & \\
& & & \vdots & & \ddots & \\
& & & & & & 0\end{array}\right]$$ 
we can define $B=2A-\omega I$, then $B$  is a matrix of order $k$.
\end{lemma}

\begin{proof}
The matrix $B$ is in the form
$$\left[\begin{array}{ccccccc} -\omega & & & & & &  \\  & \ddots & & \vdots& & &  \\
& &-\omega  & 2x & & & \\
& & & \omega & & & \\
& & & 2y & -\omega & & \\
& & & \vdots & & \ddots & \\
& & & & & & -\omega\end{array}\right]$$
and by simple verification we proof the result.
\end{proof}

Applying the proof of the Theorem \ref{th1} and the Lemma \ref{lemafiniteorder} we can enunciated the following 

\begin{corollary}\label{corol1} Let $A$
an $n\times n$ matrix and $k$ an even number. Then $A$ is the sum of finitely many matrices of order $k$ if and only if there is an polynomial $F(x) $  of degree $<k-1$ and   positive integer coefficients   such that $tr(A)=F(\omega)$ with $\omega$ a kth root of unity.
\end{corollary}
\begin{proof}
$[\Leftarrow]$ The trace of an matrix of order $k$ is an polynomial with integer coefficients and the same hold for sums of these elements.
\newline
$[\Rightarrow]$ Consider $n\geq 2$ and choose an integer $m$ such that $m\geq 2-\displaystyle\frac{1}{n}F(1)$ with $m$ and $F(1)$ are even or odd at the same time. Define $S=\displaystyle\frac{1}{2}(A+m\omega I)$ and  we have that
$tr(S)=\displaystyle\frac{1}{2}F(\omega)+\displaystyle\frac{1}{2}mn\omega$. Here define $G(x)=\displaystyle\frac{1}{2}F(x)+\displaystyle\frac{1}{2}mnx$ then $tr(S)=G(\omega)$ and $G(1)\geq n> rank(A)$, by the Theorem \ref{th1} we have that $S$ is a sum of $(k+1)$-potent matrices, $S=\sum^r_{i=1} E_i$ where the matrix $E_i$ have the form of Lemma \ref{lemafiniteorder}. Define $B_i=2E_i-\omega I$ for $i=1,2,\cdots, r$ then by the same Lemma, $B_1,B_2,\cdots,B_r$ are matrices of order $k$. So 
$$\begin{array}{rl}
T=&2S-m\omega I \\
=& 2(\displaystyle\sum^r_{i=1}E_i)-m\omega I \\ =& \displaystyle\sum^r_{i=1} (2E_i-\omega I)+(r-m)\omega I  \\
=& \displaystyle\sum^r_{i=1}B_i+(r-m)\omega I,\end{array}$$
is a sum of matrices of order $k$.
\end{proof}

And, using the Theorem \ref{th4} and the last Corollary we get the following corollary for finite order matrices:

\begin{corollary}\label{corolfiniteorder}
Let $A$ be a complex square matrix. Then $A$ is the sum of finitely many matrices of finite (different) order if and only if 
\begin{itemize}
    \item [1.)] There are $\omega_1,\omega_2,\cdots,\omega_n$ and $\beta_1,\beta_2,\cdots,\beta_n$ all even integer numbers with  $\omega_i\neq \omega_j$ and $\beta_i\neq \beta_j$ for $i\neq j$ such that $\omega_i$ is a $\beta_i th $ root of unity for $i=1,2,\cdots,n$.
    \item[2.)] There are polynomials $F_1(x),F_2(x),\cdots, F_n(x)$ with positive coefficients integers and $1\leq degree(F_i) \leq \beta_i-1$, such that $$tr(A)=a_0+F_1(\omega_i)+F_2(\omega_2)+\cdots + F_n(\omega_n)$$
    where $a_0$ is a positive integer number.
    \end{itemize}
    
\end{corollary}
\begin{proof}
The proof follows immediately from the proofs of  Theorem \ref{th4} and  Corollary \ref{corol1}.
\end{proof}
In the following Lemmas we show some results about  $k$-potent matrices.

\begin{lemma}\label{lema1}
Consider $\omega$ a kth root of unity in $F$. Then the matrices
$$B=\left[\begin{array}{cc} a & -a \\ a-\omega & \omega-a\end{array}\right] \  \ \wedge \  \ C=\left[\begin{array}{cc} a & a \\ \omega-a & \omega-a\end{array}\right],$$
satisfy the equations $X^k\neq X$ and $X^{k+1}=X$. Then both are $(k+1)$-potent matrices.
\end{lemma}
\begin{proof}
Its no difficult proof that, for all $i<k$:
$$B^{i}=\left[\begin{array}{cc} a\omega^{i-1} & -a\omega^{i-1} \\ a\omega^{i-1}-\omega^i & \omega^i-a\omega^{i-1}\end{array}\right].$$
And for $C$ we have an identity similar.

\end{proof}
\begin{lemma}\label{lema2}
If $x \in F$, where $F$ is a field, then the matrix $$A= \left[\begin{array}{cc} x & 0 \\ 0 & 2\omega-x\end{array}\right]$$
is the sum of two $(k+1)$-potent matrices.
\end{lemma}
\begin{proof}
By the Lemma \ref{lema1} is no difficult see that the matrix $A$ is the sum of matrices $B$ and $C$ with $a=\displaystyle\frac{x}{2}.$
\end{proof}

\begin{lemma}\label{lema3} Let $A$ an matrix and $\omega$ a non-trivial kth root of unity, then $(\omega I+A)\oplus (\omega I-A)$ is a sum of two $(k+1)$-potent matrices.
\end{lemma}
\begin{proof}
Consider $X=\omega I+A$ then 
$$(\omega I+A)\oplus (\omega I-A)=\left[\begin{array}{cc}X & 0 \\ 0 & 2\omega I-X\end{array}\right]$$
and the result follows from Lemma \ref{lema2}.
\end{proof}
With the above lemmas we obtain the following proposition.
\begin{proposition}\label{propo1} Consider $A$ an matrix $n\times n$, $n\geq 2$ with $tr(A)=\alpha\omega^j$ for some $\alpha>3$ an integer number, $\omega$ an non-trivial kth root of unity and $1<j<k$. Suppose that $rank(A)=1$, then \begin{itemize}
    \item [1.)] If $\alpha$ is an even number, then $A$ is a sum of $\displaystyle\frac{\alpha-2}{2}+2$,  $(k+1)$-potent matrices.
    \item[2.)] If $\alpha$ is an odd number, then $A$ is a sum of $\displaystyle\frac{\alpha-3}{2}+3$,  $(k+1)$-potent matrices.
    \end{itemize}
\end{proposition}

\begin{proof}
In the case that $\alpha$ is an even number, then $A$ is similar to 

$$\left[\begin{array}{cccc}\alpha\omega^j &  & & 0  \\  & 0 & &  \\
 & & \ddots &  \\
 0 & & & 0\end{array}\right]$$
 Then, we have
 
 $$\left[\begin{array}{cccc}\alpha\omega^j &  & & 0  \\  & 0 & &  \\
 & & \ddots &  \\
 0 & & & 0\end{array}\right]=\left[\begin{array}{cccc}\beta\omega^j &  & & 0  \\  & \beta\omega^j & &  \\
 & & \ddots &  \\
 0 & & & 0\end{array}\right]+\left[\begin{array}{cccc}(\alpha-\beta)\omega^j &  & & 0  \\  & -\beta\omega^j & &  \\
 & & \ddots &  \\
 0 & & & 0\end{array}\right],$$
 where $\beta=\displaystyle\frac{\alpha-2}{2}$. The first matrix is a sum  of $\beta$, $(k+1)$-potent matrices, and by the Lemma \ref{lema3}, the second matrix is a sum of two $(k+1)$-potent matrices. In the case that $\alpha$ is an odd number we can observe that
 
 $$\left[\begin{array}{cccc}\alpha\omega^j &  & & 0  \\  & 0 & &  \\
 & & \ddots &  \\
 0 & & & 0\end{array}\right]=\left[\begin{array}{cccc}\omega^j &  & & 0  \\  & 0 & &  \\
 & & \ddots &  \\
 0 & & & 0\end{array}\right]+\left[\begin{array}{cccc}(\alpha-1)\omega^j &  & & 0  \\  & 0 & &  \\
 & & \ddots &  \\
 0 & & & 0\end{array}\right],$$
 where the first matrix is $(k+1)$-potent and, by (1), the second matrix is a sum of $\displaystyle\frac{\alpha-3}{2}+2$,  $(k+1)$-potent matrices.

\end{proof}

In the geral case, we can enunciated the following proposition

\begin{proposition}\label{propogeral}
Consider $A$ an matrix $n\times n$, $n\geq 2$, $\omega$ an non-trivial kth root of unity and 
$tr(A)=\alpha_0+\alpha_1\omega+\alpha_2\omega^2+\cdots +\alpha_{k-1}\omega^{k-1}$ with $\alpha_i>3$ for all $i=0,1,\cdots,k-1$. If $rank(A)=k$  then $A$ is a sum of $(\sum^{k-1}_{i=0}\gamma_i+2)k$,  $(k+1)$-potent matrices where $$\gamma_i=\left\{\begin{array}{ll}\displaystyle\frac{\alpha_i-2}{2} &,  \ \textrm{if}\ \alpha_i \textrm{ is an even number}, \\
\displaystyle\frac{\alpha_i-3}{2}+3 &,\  \textrm{if} \ \alpha_i  \textrm{ is an odd number.}
\end{array}\right.$$
\end{proposition}
\begin{proof}
Consider $tr(A)=F(\omega)$ then $F(1)=\sum^{k-1}_{i=0}\alpha_i\geq rank(A)=k$, then by Fillmore in \cite{Fillmore}, $A$ is similar to

$$\left[\begin{array}{ccc}\left[\begin{array}{ccccc}\alpha_0 &  & & & 0  \\  & \alpha_1\omega & &  & \\
 & & \alpha_2\omega^2 & & \\
  & & & \ddots & \\ & & & & \alpha_{k-1}\omega^{k-1}\end{array}\right]_{k\times k} & & \\ & \ddots & \\ & & 0\end{array}\right]$$
And, using the Proposition \ref{propo1} in each element in the diagonal we concluded the proof.

\end{proof}

\begin{remark} Follows of the proof of Proposition \ref{propogeral} that if $tr(A)=\alpha_0+\alpha_1\omega+\cdots +\alpha_{k-1}\omega^{k-1}$ with $0<j<k-1$ of $\alpha_i^{\prime}s$ are zero and $k-j$ are $>3$. If $rank(A)=k-j$, then $A$ is a sum of $\displaystyle\sum^{k-1}_{i=0 \wedge \alpha_i\neq 0}\gamma_i(k-j)$, $(k+1)$-potent matrices. 
\end{remark}

Using the above proposition we can enunciated the more geral case.

\begin{proposition}\label{superpropo}
Consider $A$ an matrix $n\times n$, $ n\geq 2$, and suppose that

\begin{itemize}
    \item [1.)] There are $\omega_1,\omega_2,\cdots,\omega_n$ and $\beta_1,\beta_2,\cdots,\beta_n$ integer numbers with  $\omega_i\neq \omega_j$ and $\beta_i\neq \beta_j$ for $i\neq j$ such that $\omega_i$ is a $\beta_i th $ root of unity for $i=1,2,\cdots,n$.
    \item[2.)] There are polynomials $F_1(x),F_2(x),\cdots, F_n(x)$ with $1\leq degree(F_i) \leq \beta_i-1$ such that  $F_i(x)=\alpha_{i1}x+\alpha_{i2}x^2+\cdots+\alpha_{i(\beta_i-1)}x^{\beta_i-1}$ with all $\alpha_{ij}$ positive coefficients integers and $$tr(A)=a_0+F_1(\omega_i)+F_2(\omega_2)+\cdots + F_n(\omega_n)$$
    where $a_0$ is a positive integer number.
\end{itemize}
    If $rank(A)=1+\sum^{n}_{i=1}(\beta_i-1)$, then $A$ is a sum of $\left\{\gamma_0+\sum^{n}_{i=1}\sum^{\beta_{i-1}}_{j=1}(\gamma_{ij}+2)\right\}\times rank(A)$ many $\beta_1$, $\beta_2$, $ \cdots$, $\beta_n$-potent matrices where
    $$\gamma_0=\left\{\begin{array}{ll}\displaystyle\frac{\alpha_0-2}{2} &,  \ \textrm{if}\ \alpha_0 \textrm{ is an even number}, \\
\displaystyle\frac{\alpha_0-3}{2}+3 &,\  \textrm{if} \ \alpha_0  \textrm{ is an odd number.}
\end{array}\right.$$
and, for $i=i,2,\cdots,n$ and $j=1,2,\cdots, \beta_i$,
$$\gamma_{ij}=\left\{\begin{array}{ll}\displaystyle\frac{\alpha_{ij}-2}{2} &,  \ \textrm{if}\ \alpha_{ij} \textrm{ is an even number}, \\
\displaystyle\frac{\alpha_{ij}-3}{2}+3 &,\  \textrm{if} \ \alpha_{ij}  \textrm{ is an odd number.}
\end{array}\right.$$
\end{proposition}

\begin{proof}
Follows immediately from Theorem \ref{th4} and the Proposition \ref{propogeral}.
\end{proof}

Now, we can observe the following situations.

\begin{example}\label{ex1}
Consider $A$ a complex matrix $3\times 3$ such that $Trace(A)=4-6\omega+10\omega^2$ with $\omega$ a $3$th root of unity. Suppose that $rank(A)=3$ then we can expressed $A$ as a sums and differences of $4$-potent matrices. This is possible because $A$ is similar to 
$$B=\left[\begin{array}{ccc} 4 & & * \\ & -6\omega &  \\ * & & 10\omega\end{array}\right]$$
and
$$B=4\left[\begin{array}{ccc}1 &0 & 0 \\ * & 0& 0  \\ * &0 &0 \end{array}\right]-6\left[\begin{array}{ccc}0& * & 0 \\ 0& \omega & 0 \\ 0 &* & 0\end{array}\right]+10\left[\begin{array}{ccc}0 &0 & * \\0 & 0& *  \\ 0 & 0 & \omega\end{array}\right].$$
Then we can conclude the our affirmation since the first matrix is idempotent and the second and third matrices are $4$-potent
\end{example}

\begin{example}
If in the Example \ref{ex1} we suppose that $Trace(A)=4-\displaystyle\frac{9}{2}\omega+\displaystyle\frac{25}{4}\omega^2$ then we obtain that $B$ is a combination linear of one idempotent and two $4$-potent matrices.
\end{example}

Using the above example we can easily extend all results for matrices with all coefficients that appears in their respective trace no necessarily positives. First, consider $F(x)=\alpha_0+\alpha_1 x+\cdots \alpha_n x^n $ an polynomial with non-zero coefficients, define the module of polynomial $F$ as $|F|=|a_0|+|a_1|+\cdots +|a_n|=\sum^n_{i=0}|a_i| $ and the number  $Min(F)=\min{\{|a_j|\}}_{j=0}^n$

\begin{corollary}\label{corolgeral}
Let $A$ be a complex square matrix and suppose that
\begin{itemize}
    \item [1.)] There are $\omega_1,\omega_2,\cdots,\omega_n$ and $\beta_1,\beta_2,\cdots,\beta_n$ integer numbers with  $\omega_i\neq \omega_j$ and $\beta_i\neq \beta_j$ for $i\neq j$ such that $\omega_i$ is a $\beta_ith $ root of unity for $i=1,2,\cdots,n$.
    \item[2.)] There are polynomials $F_1(x),F_2(x),\cdots, F_n(x)$ with non-zero coefficients, $Min(F_i)\geq 1$  and $1\leq degree(F_i) \leq \beta_i-1$ for $i=1,2\cdots,n$, such that $$tr(A)=a_0+F_1(\omega_i)+F_2(\omega_2)+\cdots + F_n(\omega_n)$$
    where $a_0$ satisfies $|a_0|\geq 1$.
    \end{itemize}
    Then, $A$ is a linear combination of finitely many $\beta_1+1,\beta_2+1,\cdots,\beta_n+1$-potent matrices if and only if $$|a_0|+|F_1|+|F_2|+\cdots+|F_n|\geq rank(A) $$
\end{corollary}
\begin{proof}
Follow immediately from Theorem \ref{th4} and the above examples.
\end{proof}
And so is possible enunciate similar results for the case of linear combination of matrices of  finite order (see Corollary \ref{corolfiniteorder}) and count the quantity necessary of $k$-potent matrices for obtain a linear combination in special cases for the rank of $A$ (see Proposition \ref{superpropo}). We leave this work for the reader.

\section{Sums of k-potent matrices in $\mathcal{M}_{\mathcal{C}f}(F)$}

In this section we present some results that generalized the paper of Slowik in \cite{SlowikIdemp}.  We write $e_{\infty}$  and $e_k$ for $\mathbb{N}\times\mathbb{N}$ and $k\times k$ identity matrices, respectively, and $e_{nm}$ for the $\mathbb{N}\times\mathbb{N}$ matrix with 1 in the position $(n,m)$ and 0 in every other position.
 
Let $\mathcal{M}_{\mathcal{C}f} (F)$ be the set of all column-finite $\mathbb{N} \times  \mathbb{N}$ matrices over a field $F$. Denote   $\mathcal{T}_{ \infty} (F)$  the subring of $\mathcal{M}_{\mathcal{C}f} (F)$ consisting of all upper triangular matrices and $\mathcal{L}T_{\mathcal{C}f} (F)$ the subring of all column-finite lower triangular matrices. Here $\mathcal{T}_n(F)$ and $\mathcal{L}T_n(F)$ will be used for the rings of all $n \times  n$ upper or lower triangular matrices respectively, whereas $\mathcal{D}_n(F)$ will denote the ring of all $n \times  n$ diagonal matrices. The full $n \times  n$ matrix ring will
be denoted by $\mathcal{M}_{n\times n}(F)$.

\subsection{Triangular matrices}

Consider $\omega \in F$ a kth root of unity, $w\neq 1$. Here we have the following lemma

\begin{lemma}\label{lema4.1} Let $F$ a field and let $t \in \mathcal{T}_{\infty}(F)$. If
$$t=\omega e_{\infty}+ \sum_{m-n\geq 1}t_{nm}e_{nm}$$
and $t_{n,n+1}\neq 0$ for all $n\in \mathbb{N}$, then $t$ is similar to
$$\omega e_{\infty}+\sum^{\infty}_{n=1}t_{n,n+1}e_{n,n+1}.$$
\end{lemma}
\begin{proof} Let $t=\omega e_{\infty}+ \sum_{m-n\geq 1}t_{nm}e_{nm}$. Define $u=\displaystyle\frac{1}{\omega}\cdot t$, in this case $u$ are in the conditions of Lemma  2.1 in \cite{SlowikIdemp},
so $u$ is similar to $$ e_{\infty}+\sum^{\infty}_{n=1}\frac{1}{\omega}\cdot t_{n,n+1}e_{n,n+1}.$$
Therefore, $t$ is similar to $$\omega e_{\infty}+\sum^{\infty}_{n=1}t_{n,n+1}e_{n,n+1},$$
what demonstrates the Lemma.
\end{proof}
Next, proof the following proposition

\begin{proposition}\label{prop1}Let $F$ be a field. If $t\in \mathcal{T}_{\infty}(F)$ satisfies the condition $t_{nn}=2\omega$ for all $n \in \mathbb{N}$, then $t$ is a sum of at most four $(k+1)$-potent matrices.
\end{proposition}
\begin{proof} Following the Slowik demonstration of Lemma  2.1 in \cite{SlowikIdemp}. First we define two matrices $t_1$ and $ t_2$ as follows:
$$
(t_1)_{nm}=\left\{\begin{array}{lcrcccr}
\omega & \text{if} & m=n, &&&&\\
t_{nm} & \text{if} & m-n>1, &&&&\\
t_{nm} & \text{if} & m-n=1, &\  and &t_{nm}\neq0,&& t_2=t-t_1\\
\omega & \text{if} & m-n=1, &\ and &t_{nm}=0,&&
\end{array}
\right.$$

One can see that $t_1$ fulfills the assumptions of Lemma \ref{lema4.1}. Hence, $t_1$ is similar to the matrix
$$\omega e_{\infty}+\sum^{\infty}_{n=1}(t_1)_{n,n+1}e_{n,n+1},$$
and we write this matrix as the sum $u + u'$, where
 $$u=\sum_1^{\infty} \left(\omega e_{2n-1,2n-1}+(t_1)_{2n-1,2n}e_{2n-1,2n}\right)\ \text{and}\ \ u'=\sum_1^{\infty}\left(\omega e_{2n,2n}+(t_1)_{2n,2n+1}e_{2n,2n+1}\right),$$
it is easy to check that $u$ and $u'$ are $(k+1)$-potent matrices.

Now we will write $t_2 = v'+v''$ where $v'$ is defined by the following inductive rule:
\begin{itemize}
    \item[(1)] $v'_{11}=(t_2)_{11},$ $v'_{12}=(t_{2})_{12}$;
    \item[(2)] if $ v'_{n,n+1} = 0$, then put we $v'_{n+1,n+1} = \omega$, $v'_{n+1,n+2} = (t_2)_{n+1,n+2}$;
    \item[(3)] if $v'_{n,n+1} \neq  0$, then we put $v'_{n+1,n+1} = 0,v'_{n+1,n+2} = 0$.
\end{itemize}

 This it, if $v'_{n,n+1} \not = 0$, then the entries from the next row 'go' to $v''$. If $v'_{n,n+1} = 0$, then the next row can 'stay' in $v'$.
 
 From the construction of $v'$ and $v''$ it follows that they are $(k+1)$-potent matrices. 
 
\end{proof}

\begin{corollary}\label{cor1} Let $F$ be a field and $k \in \mathbb{N}$ . If $t$ is either from $\mathcal{T}_{\infty}(F)$ or $\mathcal{L}T_k(F)$ and $t_{nn} = 2\omega$ for all $n$, $1 \leq n \leq k$, then $t$ is a sum of at most four $(k+1)$-potent matrices.
\end{corollary}
 
 The above corollary will be useful in the proof of the below result.

 \begin{proposition}\label{prop2} If $F$ is a field and $t \in \mathcal{L}T_{Cf}(F)$ is such that $t_{nn} = 3\omega$ for all $n \in\mathbb{N}$, then $t$
can be written as a sum of at most six $(k+1)$-potent matrices.
 \end{proposition}
 
 \begin{proof}
 We apply the method used in the proof of Proposition 2.2 from \cite{SlowikIdemp}. First we construct the sequence $(l_m)_{m\in\mathbb{N}}$ as follows:
 $$l_m = \max \{ i: t_{im} \neq 0\}.
 $$
Now we define another sequence $(l'_m)_{m\in \mathbb{N}}$ as follows:
$$
l'_m=\left\{\begin{array}{ll}
l_m & \ if \ for \ all \ j<m\  we\ have\ l_j\leq l_m,\\
\max \{ l_j : j < m\ \wedge\ l_j > l_m\} & \ otherwise
\end{array}
\right.
$$

The sequence $(l'_m)_{m\in \mathbb{N}}$ is nondecreasing. Observed that for all $m \in \mathbb{N}$ and all $i > l'_m$ we
have $t_{im} = 0$. Thus, we can say that $t$ has staircase structure and that the stairs are determined by
the sequence $(l'_m)$. 

Now we will write $t=u+v$ where the matrix $u$ is constructed as follows.
For all $i \in \mathbb{N}$ we put $u_{ii} = 2\omega$. Now let $m$ be equal to $1$ and let $n$ be equal to $l'_1$. For all $i, j$ such
that $m \leq j < i \leq n$ we put $u_{ij} = t_{ij}$.

Let now $m$ be equal to the preceding $n$ increased by $1$ and let the new $n$ be equal to $l'_m$. Then for all $i, j, m \leq i, j \leq n$ we put $u_{ij} = t_{ij}$ .

Proceeding the same way we obtain a block diagonal matrix $u$ and consider $v = t - u$. By Corollary \ref{cor1} each block of $ u$ is a sum of at most four $(k+1)$-potent matrices. From the construction of $v$ it follows that

$$v=\left[\begin{array}{cccccc}
\omega & & & & & \\
v_1 &\omega&&&&\\
0&v_2&\omega&&&\\
0&0&v_3&\omega&&\\
0&0&0&v_4&\omega&\\
 \vdots&\vdots&&&\ddots&\ddots
\end{array}
\right] 
$$
and
$$
v'=\left[\begin{array}{cccccc}
\omega & & & & & \\
v_1 &0&&&&\\
0&0&\omega&&&\\
0&0&v_3&0&&\\
0&0&0&0&\omega&\\
 \vdots&\vdots&&&\ddots&\ddots
\end{array}
\right] \ \ \ \ \
v'' = v - v'=\left[\begin{array}{cccccc}
0 & & & & & \\
 0&\omega&&&&\\
0&v_2&0&&&\\
0&0&0&\omega&&\\
0&0&0&v_4&0&\\
 \vdots&\vdots&&&\ddots&\ddots
\end{array}
\right].
$$
It is easy to check that $v', v''$ are $(k+1)$-potent matrices.

 \end{proof}



\subsection{Diagonal Matrices}
For $\mathcal{D}_{\infty}(F)$ we have the following proposition:

\begin{proposition}\label{prop3}
Let $F$ any field. Any $D\in \mathcal{D}_{\infty}(F)$ is the sum of four $(k+1)$-potent matrices from $\mathcal{M}_{\mathcal{C}f}(F).$
\end{proposition}
\begin{proof}
The proof follows from Slowik in \cite{SlowikIdemp}. For $D=diag(d_{11},d_{22},d_{33},\cdots )$ define the matrices $X,Y\in \mathcal{D}_{\infty}(F)$ inductively as

$$\begin{array}{ll}
x_{11}=d_{11}, & y_{11}=0, \\
x_{22}=2\omega-x_{11}, & y_{22}=d_{22}-x_{22}, \\
y_{33}=2\omega-y_{22}, & x_{33}=d_{33}-y_{33}, \\
x_{44}=2\omega-x_{33}, & y_{44}=d_{44}-x_{44}, \\
\vdots & \vdots
\end{array}$$
Then, $X$ and $Y$ are in the form

$$X=\left[\begin{array}{cc} x_{11} & 0 \\ 0 & 2\omega-x_{11}\end{array}\right] \oplus \left[\begin{array}{cc} x_{33} & 0 \\ 0 & 2\omega-x_{33}\end{array}\right] \oplus \cdots 
$$ 

$$Y=[0]\oplus \left[\begin{array}{cc} y_{22} & 0 \\ 0 & 2\omega-y_{22}\end{array}\right] \oplus \left[\begin{array}{cc} y_{44} & 0 \\ 0 & 2\omega-y_{44}\end{array}\right] \oplus \cdots 
$$
and, by the Lemma \ref{lema2}, both $X$ and $Y$ are the sums of $(k+1)$-potent matrices.\end{proof}

Now, we can prove our following result.
\begin{proof}[Proof of Theorem \ref{th2}]
Define $t_1, t_2$ and $d$ as follows:
$$
(t_1)_{ij}=\left\{\begin{array}{lcl} 
2\omega & \text{if} & j=i\\ 
a_{ij} & \text{if} & i<j \\
0 &\text{if}&i>j
\end{array}
\right.
\ 
(t_2)_{ij}= \left\{\begin{array}{lcl} 
3\omega & \text{if} & j=i\\ 
a_{ij} & \text{if} & i>j \\
0 &\text{if}&i<j
\end{array}
\right.
\ 
d_{ij}= \left\{\begin{array}{lcl} 
a_{ij} -5\omega & \text{if} & j=i\\ 
0 & \text{if} & i\neq j
\end{array}
\right.$$

Clearly $A = t_1 + t_2 + d$. From Propositions \ref{prop1} and \ref{prop2} we know that $t_1$ and $t_2 $ are sums of at most
four and six $(k+1)$-potent matrices, respectively. By Proposition \ref{prop3} the matrix $d$ is a sum of four $(k+1)$-potent matrices. This proves the claim.
\end{proof}

\end{document}